\documentclass[11pt]{article}
\usepackage{amsmath, amsthm, amstext}
\usepackage{amssymb, array, amsfonts}
\usepackage[utf8]{inputenc}
\usepackage[english]{babel}
\usepackage{enumerate}
\usepackage{graphicx}
\usepackage{mathrsfs}

\usepackage{cite}

\inputencoding{utf8}

\numberwithin{equation}{section}

\newtheorem{thm}{Theorem}
\newtheorem{lemma}{Lemma}

\title{Asymptotic properties of solutions to the characteristic problem for the ultrahyperbolic equation}
\author{Maxim N. Demchenko\footnote{St.~Petersburg Department of
V.\,A.~Steklov Institute of Mathematics of
the Russian Academy of Sciences, 
27 Fontanka, St.~Petersburg, Russia. E-mail: demchenko@pdmi.ras.ru}}
\date{}

\begin{document}
\maketitle

\begin{abstract}
  The paper concerns the problem for the ultrahyperbolic equation in the Euclidean space with data on a characteristic hyperplane.
  Smoothness and asymptotics of the solution along characteristic lines transversal to the initial hyperplane are investigated.
\medskip


\end{abstract}

\section{Introduction}
We consider the problem for the ultrahyperbolic equation
\begin{equation}
  (\partial^2_{ts} + \partial_{x_1}^2 +\ldots + \partial_{x_d}^2 - \partial_{y_1}^2 - \ldots - \partial_{y_n}^2) v = 0,
  \label{eqn-psi}
\end{equation}
where the solution $v$ is a function of variables
\begin{equation*}
  (t, s, x_1, \ldots, x_d, y_1, \ldots, y_n) \in 
  {\mathbb R}^{N+2}, \quad N = d + n,
  \quad
  d,n \geqslant 1,
\end{equation*}
subject to the following condition
\begin{equation}
  v|_{t=0} = v_0.
  \label{cauchy}
\end{equation}
The latter is imposed on the hyperplane $\{t=0\}$, which,
as will be seen later, is a {\em characteristic} one for equation~(\ref{eqn-psi}).
We will consider solutions $v$ from the class of $L_2({\mathbb R}^{N+1})$-valued continuous functions of variable $t\in{\mathbb R}$.
The meaning of condition~(\ref{cauchy}) for such functions and data $v_0\in L_2({\mathbb R}^{N+1})$ is obvious.
We will demand that equation~(\ref{eqn-psi}) be satisfied in the whole space ${\mathbb R}^{N+2}$ in the sense of distributions.
This also makes sense, as solutions from the class in question can be considered as elements of $L_{2,\rm loc}({\mathbb R}^{N+2})$.

Problem~(\ref{eqn-psi}), (\ref{cauchy}) was first studied in~\cite{Blag}.
The main result obtained there was a representation of the solution as a convolution of the data $v_0$ with a certain distribution
in ${\mathbb R}^{N+1}$. 
The latter was given in terms of analytic continuation of a family of distributions with respect to a complex parameter.
The present paper concerns the issues of smoothness of the solution and its asymptotic behavior for large $t$.
We will use the representation of the solution obtained in~\cite{char-setup}
(where the well-posedness of the problem was also established).

Our interest in the characteristic problem for equation~(\ref{eqn-psi}) is motivated by well-posedness of such a problem
in standard (non-analytic) classes of functions,
in contrast to the Cauchy problem with data on a noncharacteristic hypersurface.
As is well known, in the case of a hyperbolic equation, the Cauchy problem is well-posed exactly when
the initial data are specified on a {\em spacelike} hypersurface.
However, the latter has no counterparts in the case of ultrahyperbolic equation,
which, in essence, is the reason of ill-posedness of the corresponding Cauchy problem. 

The following conservation law for problem~(\ref{eqn-psi}), (\ref{cauchy}) was established in~\cite{Blag}: 
\begin{equation*}
  \|v(t, \cdot)\|_{L_2} = \|v_0\|_{L_2}, \quad t\in{\mathbb R}.
\end{equation*}
This feature, together with the well-posedness, 
demonstrates the similarity of the problem with classical evolution problems
for hyperbolic equations.
In particular, it seems to be achievable to develop an analogue of stationary scattering theory for such a problem
considering $t$ as a time parameter.
One of the corner stones of such a theory would be the study of the asymptotic behavior of solutions for large $t$,
which is the subject of the present paper.

The majority of results in this field either concern corresponding fundamental solutions
(we mention the pioneering work~\cite{deRham} and more recent ones~\cite{HolstMelin, Ortner-Wagner})
or generalize Asgeirsson's theorem~\cite{Lyakhov-Polovinkin-Shishkina, Lyakhov-Polovinkin-Shishkina-DAN}.
In the well-posed problems for ultrahyperbolic equations known so far, the data are specified on a characteristic hypersurface or at the infinity
(the scattering data).
Besides the results of~\cite{Blag} mentioned above,
we also mention the paper~\cite{Blag-char},
where the well-posedness of the problem in Euclidean space with data on the characteristic cone was established.
Problems with data at the infinity were considered in~\cite{DMN23}.
In the case of hyperbolic equations, such kind of problems were studied first in~\cite{Blag-scatter, Lax-Phillips}
and later in~\cite{Moses, Kis, Plachenov}.
As was mentioned above, the Cauchy problem for the ultrahyperbolic equation is ill-posed in general.
However, there are positive results in the case, when the solution is assumed to be symmetric with respect to a part of variables~\cite{Lyakhov-Bulatov}
or the initial data belong to a special class of functions~\cite{Weinstein}.
Among other nonclassical problems close to the one considered in the present paper,
the most complete investigation was carried out for the problems connected with {\em pseudohyperbolic} equations~\cite{Demid}.

\section{Smoothness and asymptotics of the solution to problem~(\ref{eqn-psi}), (\ref{cauchy})}
The result of the paper concerns the smoothness of the solution to problem~(\ref{eqn-psi}), (\ref{cauchy}) for $t\ne 0$
and its behavior for large $t$.
As in the case of hyperbolic equations~\cite{Blag-scatter, Lax-Phillips},
asymptotics of solutions along characteristic lines transversal to the initial hyperplane will be established.

To formulate the result, 
it is convenient to pass from variables $t, s$ to
\begin{equation*}
  x_0 = t+s, \quad y_0 = t-s.
\end{equation*}
Put
\begin{gather*}
  \overline x = (x_1, \ldots, x_d) \in {\mathbb R}^d,
  \quad
  \overline y = (y_1, \ldots, y_n) \in {\mathbb R}^n,
  \\
  x = (x_0, \overline x) \in {\mathbb R}^{d+1},
  \quad
  y = (y_0, \overline y) \in {\mathbb R}^{n+1}.
\end{gather*}
Since
\begin{equation*}
  \partial_t = \partial_{x_0} + \partial_{y_0}, \quad \partial_s = \partial_{x_0} - \partial_{y_0},
  \quad
  \partial_t \partial_s = \partial_{x_0}^2 - \partial_{y_0}^2,
\end{equation*}
the function
\begin{equation}
  u(x,y) = v\left(t, s, \overline x, \overline y\right) 
  \label{uv}
\end{equation}
satisfies the following equation
\begin{gather}
  (\Delta_x - \Delta_y) u = 0, \label{eq-u}
  \\ \Delta_x = \partial_{x_0}^2 + \ldots + \partial_{x_d}^2, \quad \Delta_y = \partial_{y_0}^2 + \ldots + \partial_{y_n}^2. \notag
\end{gather}
In coordinates $(x,y)$, the condition $\{t=0\}$ determines the set 
\begin{equation}
  \{(x,y)\, |\, x_0 + y_0 = 0\},
  \label{char-plane}
\end{equation}
which is a characteristic hyperplane for equation~(\ref{eq-u}).
Therefore, the set $\{t=0\}$ is a characteristic hyperplane for equation~(\ref{eqn-psi}).

The result will be given in terms of the function $u$ related to the solution $v$ by equality~(\ref{uv}).
We will need a description of lines that are characteristic with respect to equation~(\ref{eq-u}) and transversal to hyperplane~(\ref{char-plane}).
Such lines consist of points $(x,y)$ of the form
\begin{gather}
  x = X + \tau \theta, \quad
  y = Y + \tau \omega, \label{xy-tending}
  \\ (X, Y, \theta, \omega) \in {\mathbb R}^{d+1}\times{\mathbb R}^{n+1}\times S^d\times S^n \notag
\end{gather}
($S^m$ is the unit sphere in ${\mathbb R}^{m+1}$),
where $\tau$ is a real parameter.
The transversality condition above reads
\begin{equation}
  \theta_0 + \omega_0 \ne 0.
  \label{nontang}
\end{equation}
In view of the last constraint, we may assume that the ``initial point'' $(X,Y)$ in~(\ref{xy-tending}) (which corresponds to $\tau = 0$)
lies on the hyperplane~(\ref{char-plane}),
which is equivalent to condition
\begin{equation}
  X_0 + Y_0 = 0.
  \label{char-init}
\end{equation}
Under this assumption, a point $(x,y)$ corresponding to $\tau \ne 0$ lies outside of this hyperplane.
Observe that condition~(\ref{nontang}) ensures the point $(x,y)$ with large $\tau$
corresponds to the point $(t,s,\overline x,\overline y)$ with large $t$.

Further we will need some estimates locally uniform with respect to $(x,y)$ outside of hyperplane~(\ref{char-plane}).
Any such point together with its small neighborhood can be represented in the form~(\ref{xy-tending})
with fixed $\theta,\omega$ satisfying~(\ref{nontang}) and parameters $\tau$, $X,Y$ that vary within small neighborhoods
and satisfy conditions~(\ref{char-init}) and $\tau \ne 0$.
Thus, in locally uniform estimates, constants may depend arbitrarily on $\theta, \omega, R$
providing that
\begin{equation}
  |X| + |Y| \leqslant R.
  \label{char-init-R}
\end{equation}

We will study the function $u(x,y)$ and its partial derivatives with multiindex $\beta$:
\begin{equation*}
  \partial^\beta u = \partial_{x_0}^{\beta_0} \ldots \partial_{x_d}^{\beta_d}\, \partial_{y_0}^{\beta_{d+1}} \ldots \partial_{y_n}^{\beta_{N+1}} u.
\end{equation*}
The formulation of the result involves the polynomial
\begin{equation*}
  P^\beta(\theta,\omega) = (i\theta_0)^{\beta_0} \ldots (i\theta_d)^{\beta_d} (-i\omega_0)^{\beta_{d+1}} \ldots (-i\omega_n)^{\beta_{N+1}},
\end{equation*}
and the Fourier transform of the function $v_0(s, \overline x, \overline y)$
defined as follows
\begin{equation*}
  \widetilde v_0(\lambda, \overline\xi, \overline\eta) = 2 \int_{{\mathbb R}^{N+1}} e^{-i(s\lambda + \overline x\,\overline\xi - \overline y\,\overline\eta)}
  v_0(s, \overline x, \overline y)\, ds d\overline x d\overline y,
  \quad
  (\lambda, \overline\xi, \overline\eta) \in {\mathbb R}\times {\mathbb R}^d \times {\mathbb R}^n.
\end{equation*}
Here and further expression $w z$ for $w,z\in {\mathbb R}^m$ denotes the inner product,
$w^2 := w w$.

\begin{thm}\label{thm}
  Let $v_0(s, \overline x, \overline y)$ belong to Schwartz class $\mathcal{S}({\mathbb R}^{N+1})$.
  Then the function $u(x,y)$ related to the solution of problem~(\ref{eqn-psi}), (\ref{cauchy}) via~(\ref{uv})
  is $C^\infty$-smooth outside of hyperplane~(\ref{char-plane}).
  Under assumptions~(\ref{nontang}), (\ref{char-init}), (\ref{char-init-R}), for any point $(x,y)$ of the form~(\ref{xy-tending})
  with $\tau\geqslant\tau_0>0$, we have
\begin{gather}
  \left|\partial^\beta u(x, y) - \tau^{-N/2} F^\beta(\theta, \omega, p) \right|
  \leqslant C_{R, \tau_0, \theta, \omega, \beta, v_0} \tau^{-(N+1)/2},
  \label{thm-asymp}
  \\
  F^\beta(\theta, \omega, p) :=
  P^\beta(\theta,\omega) |\theta_0 + \omega_0|
  \int_{\mathbb R} e^{i r p} G^\beta_{d,n}(r) \widetilde v_0(r(\theta_0 + \omega_0), r\overline\theta, r\overline\omega)  dr,
  \notag\\
  G^\beta_{d,n}(r) := \frac{e^{i ({\rm sgn}\, r) \pi (n-d)/4} ({\rm sgn}\, r)^{|\beta|} |r|^{N/2+|\beta|}}{2 (2\pi)^{N/2+1}}, \quad p := X\theta - Y\omega.
  \notag
\end{gather}
\end{thm}
Note that the integral in the relation for the coefficient $F^\beta(\theta, \omega, p)$ is absolutely convergent
due to condition~(\ref{nontang}) and rapid decay of the function $\tilde v_0$ at the infinity.

Asymptotics~(\ref{thm-asymp}) provides a full description of a solution in the following sense. The coefficient $F^\beta(\theta, \omega, p)$ is defined on the manifold $S^d\times S^n\times{\mathbb R}$ of the dimension $N+1$.
This coefficient turns out to be sufficient to determine the initial data $v_0$ (which are defined on the hyperplane of the same dimension) and thus the solution $v$ in the entire space ${\mathbb R}^{N+2}$~\cite{DMN-DD25}.

\section{The Fourier transform of the function $u(x,y)$}\label{sec-fourier}
It was shown in~\cite{char-setup} that for $v_0\in \mathcal{S}({\mathbb R}^{N+1})$,
the solution to problem~(\ref{eqn-psi}), (\ref{cauchy}) and thus the function $u(x,y)$ defined by~(\ref{uv})
are tempered distributions in ${\mathbb R}^{N+2}$.
For the generalized Fourier transform of $u(x,y)$ defined as
\begin{equation*}
  \hat u(\xi, \eta) = \int_{{\mathbb R}^{N+2}} e^{-i (x\xi - y\eta)} u(x, y) dx dy,
\end{equation*}
the following formula was obtained
\begin{equation}
  \hat u(\xi, \eta)
  = a(\xi, \eta) \,\delta\left(\xi^2 - \eta^2\right),
  \quad
  a(\xi, \eta) = 2\pi |\xi_0 + \eta_0|\, \widetilde v_0(\xi_0 + \eta_0, \overline\xi, \overline\eta).
  \label{uhat}
\end{equation}
In order to describe the precise meaning of the r.h.s. of the first equality in~(\ref{uhat}),
as well as for further use,
we introduce the following submanifolds of ${\mathbb R}^{N+2}$:
\begin{gather*}
  \Sigma = S^d\times S^n,
  \\
  \mathcal{C} = \left\{(\xi,\eta)\in {\mathbb R}^{N+2}\setminus\{0\}\,|\, \Upsilon(\xi,\eta) = 0\right\},
  \\
  \mathcal{C}_1 = \{(\xi,\eta)\in {\mathbb R}^{N+2}\setminus\{0\}\,|\,\, \Upsilon(\xi,\eta) = \Upsilon_0(\xi,\eta) = |\overline\xi| = |\overline\eta| = 0\},
  \\
  \mathcal{C}_0 = \{(\xi,\eta)\in {\mathbb R}^{N+2}\setminus\{0\}\,|\, \Upsilon(\xi,\eta) = \Upsilon_0(\xi,\eta) = 0, \, |\overline\xi| = |\overline\eta| \ne 0\},
\end{gather*}
where
\begin{equation*}
  \Upsilon(\xi,\eta) = \xi^2 - \eta^2,
  \quad
  \Upsilon_0(\xi,\eta) = \xi_0 + \eta_0,
\end{equation*}
Note that in the definitions of $\mathcal{C}_0$, $\mathcal{C}_1$, the condition $|\overline\xi| = |\overline\eta|$ is redundant and is given
for clarity only.
Clearly, $\mathcal{C}_0, \mathcal{C}_1 \subset \mathcal{C}$.

We will use the following diffeomorphism 
\begin{equation}
  \Sigma\times{\mathbb R}_+ \to \mathcal{C}, \quad (\zeta,\sigma,r) \mapsto (r\zeta, r\sigma).
  \label{diffeo}
\end{equation}
Introduce the volume form $\frac{1}{2} r^{N-1}dS_\zeta dS_\sigma dr$ on $\Sigma\times{\mathbb R}_+$,
where $dS_\zeta$, $dS_\sigma$ are the standard volume forms on spheres $S^d$, $S^n$, respectively.
Let $dS_{\xi,\eta}$ be the counterpart of this form induced on $\mathcal{C}$ by diffeomorphism~(\ref{diffeo}).
For further references, we record this formally as follows
\begin{equation}
  dS_{\xi,\eta} = \frac{1}{2} r^{N-1}dS_\zeta dS_\sigma dr.
  \label{dS-coord}
\end{equation}
Action of the distribution $\hat u(\xi,\eta)$ of the form~(\ref{uhat})
on a test function $\varphi(\xi,\eta)$ is defined by the equality
\begin{equation}
  \langle\hat u, \varphi\rangle =
  \int_\mathcal{C} (a\varphi)(\xi,\eta) \, dS_{\xi,\eta}.
  \label{uhat-C}
\end{equation}

In the proof of Theorem~\ref{thm}, formula~(\ref{uhat}) will be used.
Formal inversion of the Fourier transform in the latter yields
\begin{multline*}
  u(x,y) = (2\pi)^{-N-2} \int_\mathcal{C} a(\xi, \eta) e^{i (x\xi - y\eta)} dS_{\xi,\eta}
  \\= (2\pi)^{-N-1} \int_\mathcal{C} |\xi_0 + \eta_0| \, \tilde v_0(\xi_0 + \eta_0, \overline\xi, \overline\eta)
  e^{i (x\xi - y\eta)} dS_{\xi,\eta}.
\end{multline*}
For a point $(x,y)$ of the form~(\ref{xy-tending}), the last relation reads
\begin{equation}
  u(x, y) = (2\pi)^{-N-1} \int_\mathcal{C} |\xi_0 + \eta_0| \, \tilde v_0(\xi_0 + \eta_0, \overline\xi, \overline\eta)
  e^{i (X \xi - Y \eta)} e^{i \tau \Phi(\xi,\eta)} dS_{\xi,\eta},
  \label{u-as}
\end{equation}
where
\begin{equation*}
  \Phi(\xi,\eta) = \theta\xi - \omega\eta.
\end{equation*}
Note that the integral here is not absolutely convergent.
The reason is that the arguments of the function $\tilde v_0$ in the integrand may stay bounded when the point $(\xi,\eta)$
tends to infinity along $\mathcal{C}_1\subset\mathcal{C}$.
The integral converges due to oscillations of the exponential in the neighborhood of this set.
Besides, the integrand is nonsmooth at $\xi_0 + \eta_0 = 0$, i.e. on the set $\mathcal{C}_0\cup\mathcal{C}_1$.
This complicates the analysis of smoothness and asymptotic behavior of the solutions,
as compared to the case of hyperbolic equations.

\section{Littlewood-Paley decomposition of the function $u(x,y)$}\label{sec-little-paley}
In our study of the distribution $u(x,y)$, we will consider its Littlewood-Paley decomposition,
which is defined in terms of the Fourier transform $\hat u(\xi,\eta)$.
Let a function $\tilde\chi(\rho)$ satisfy
\begin{equation}
  \tilde\chi\in C_0^\infty({\mathbb R}_+), \quad
  \sum_{j\in{\mathbb Z}} \tilde\chi(2^{-j} \rho) = 1, \quad \rho>0
  \label{sum-chi}
\end{equation}
(for any positive $\rho$, the series on the l.h.s. of the last equality contains finite number of nonzero terms).
Introduce the following functions
\begin{equation*}
  \chi(\xi,\eta) = \tilde\chi\left(\sqrt{(\xi^2+\eta^2)/2}\right),
  \quad
  \chi_j(\xi,\eta) = \chi(2^{-j} \xi, 2^{-j} \eta), 
  \quad j\in{\mathbb Z}.
\end{equation*}
The restrictions $\chi_j|_\mathcal{C}$ are compactly supported functions in $\mathcal{C}$
satisfying, by properties~(\ref{sum-chi}), the following relation
\begin{equation*}
  \sum_{j\in{\mathbb Z}} \chi_j(\xi,\eta) = 1.
\end{equation*}
Since the amplitude $a(\xi,\eta)$ is a bounded function,
this equality implies that
formula~(\ref{uhat-C}) can be written as the following relation
\begin{equation*}
  \left\langle\hat u, \varphi\right\rangle = \sum_{j\in{\mathbb Z}} \int_\mathcal{C} (\chi_j a \varphi)(\xi,\eta) \, dS_{\xi,\eta},
\end{equation*}
which essentially is the Littlewood-Paley decomposition of the distribution $u(x,y)$.
This can be generalized on the partial derivatives $\partial^\beta u$ as follows
\begin{equation}
  \left\langle\widehat{\partial^\beta u}, \varphi\right\rangle = \sum_{j\in{\mathbb Z}} \int_\mathcal{C} (\chi_j a^\beta\varphi)(\xi,\eta) \, dS_{\xi,\eta},
  \label{hatu-sum}
\end{equation}
where
\begin{equation}
  a^\beta(\xi,\eta) = (P^\beta a)(\xi,\eta)
  \label{abe}
\end{equation}
(the polynomial $P^\beta$ was defined in the formulation of Theorem~\ref{thm}).
Observe that the terms of the series in~(\ref{hatu-sum}) are still well-defined,
if a test function $\varphi(\xi,\eta)$ is replaced by the function
\begin{equation*}
  (2\pi)^{-N-2} e^{i(x\xi - y\eta)}
\end{equation*}
for arbitrary $(x,y)$.
As will be shown in sec.~\ref{sec-proof}
(which will be the principal part of the proof of Theorem~\ref{thm}),
such a series is absolutely convergent locally uniformly with respect to $x,y$ outside of the hyperplane~(\ref{char-plane}).
This implies that outside this hyperplane,
the sum of this series gives exactly the value of 
the function $\partial^\beta u(x,y)$: 
\begin{equation}
  \partial^\beta u(x,y) = (2\pi)^{-N-2} \sum_{j\in{\mathbb Z}} \int_\mathcal{C} (\chi_j a^\beta)(\xi,\eta) e^{i(x\xi - y\eta)} \, dS_{\xi,\eta}.
  \label{u-xy}
\end{equation}
To see this, consider an arbitrary test function $\psi(x,y)$ with a compact support outside of the hyperplane~(\ref{char-plane}).
By substituting the expression for its inverse Fourier transform
\begin{equation*}
  \varphi(\xi,\eta) = (2\pi)^{-N-2} \int_{{\mathbb R}^{N+2}} e^{i(x\xi - y\eta)} \psi(x, y) \,dx dy
\end{equation*}
into the r.h.s. of~(\ref{hatu-sum}), we obtain
\begin{equation*}
  \left\langle\partial^\beta u, \psi\right\rangle = \left\langle\widehat{\partial^\beta u}, \varphi\right\rangle
  = (2\pi)^{-N-2} \sum_{j\in{\mathbb Z}} \int_{{\mathbb R}^{N+2}} dx dy\, \psi(x, y) \int_\mathcal{C} (\chi_j a^\beta)(\xi,\eta) e^{i(x\xi - y\eta)} \, dS_{\xi,\eta}
\end{equation*}
(we have changed the order of integration with respect to $x,y$ and $\xi,\eta$).
This is equivalent to
\begin{equation*}
  \left\langle\partial^\beta u, \psi\right\rangle 
  = (2\pi)^{-N-2} \int_{{\mathbb R}^{N+2}} dx dy\, \psi(x, y) \sum_{j\in{\mathbb Z}} \int_\mathcal{C} (\chi_j a^\beta)(\xi,\eta) e^{i(x\xi - y\eta)} \, dS_{\xi,\eta},
\end{equation*}
since the function $\psi(x,y)$ is compactly supported, and the series under the integral sign
converges uniformly with respect to $(x,y) \in {\rm supp}\, \psi$.
This implies equality~(\ref{u-xy}).

It is convenient to transform the terms of the series~(\ref{u-xy}) as follows.
By taking $2^{-j} \xi$, $2^{-j}\eta$ as new variables of integration
and invoking the definition of the volume form~(\ref{dS-coord}),
we arrive at the equivalent relation
\begin{equation}
  \partial^\beta u(x,y) = \sum_{j\in{\mathbb Z}} (2\pi)^{-N-2}\, 2^{N j} I^\beta_j(x,y), 
  \label{sum-CC}
\end{equation}
where
\begin{equation}
  I^\beta_j(x,y) = \int_\mathcal{C} a^\beta_j(\xi, \eta) e^{i 2^j(x\xi - y\eta)} \,dS_{\xi,\eta},
  \quad
  a^\beta_j(\xi,\eta) = \chi(\xi,\eta) a^\beta(2^j \xi, 2^j \eta).
  \label{ajbe}
\end{equation}

Remarks concerning the integral~(\ref{u-as}) made in the end of sec.~\ref{sec-fourier}
are relevant 
to the integrals $I^\beta_j$:
the integrand does not decay for large $j$ and $(\xi,\eta)\in \mathcal{C}_1$,
and has a singularity at $(\xi,\eta)\in \mathcal{C}_0\cup\mathcal{C}_1$.
For this reason, we will decompose $I^\beta_j$ to the sum of the integrals of the form
\begin{equation*}
  I^\beta_{\varkappa,j}(x,y) = \int_\mathcal{C} (\varkappa a^\beta_j)(\xi, \eta) e^{i 2^j (x\xi - y\eta)} dS_{\xi,\eta}
\end{equation*}
with smooth cut-off functions $\varkappa$.
We will consider the cases when these functions are supported near $\mathcal{C}_0$, $\mathcal{C}_1$,
and conclude that the corresponding integrals do not contribute to the principal term of asymptotics~(\ref{thm-asymp}).
We will see that the latter is determined by stationary points
of the restriction of the phase function $\Phi(\xi,\eta)$ to $\mathcal{C}$.
These points constitute the rays
\begin{equation}
  \mathcal{C}_{\rm st} = \{ (\pm r\theta, \pm r\omega) \,|\, r>0\} \subset \mathcal{C}\setminus(\mathcal{C}_0\cup\mathcal{C}_1)
  \label{stat-ray}
\end{equation}
(the signs $\pm$ are assumed to be equal).
The last inclusion follows from condition~(\ref{nontang}).

\section{Auxiliary assertions}
Note that the differentials $d\Upsilon$, $d\Upsilon_0$ are nonzero everywhere in ${\mathbb R}^{N+2}\setminus\{0\}$.
Besides, they are linearly independent at points $(\xi,\eta) \in \mathcal{C}_0$,
since the contrary means that $d\Upsilon = \alpha d\Upsilon_0$,
whence
\begin{equation*}
  \overline\xi = \frac{1}{2}\partial_{\overline\xi}\Upsilon = \frac{\alpha}{2}\partial_{\overline\xi}\Upsilon_0 = 0,
\end{equation*}
which contradicts the definition of the set $\mathcal{C}_0$.

Thus the system of equations $\Upsilon = \Upsilon_0 = 0$,
which locally determines the set $\mathcal{C}_0$, is nondegenerate,
and so $\mathcal{C}_0$ is a smooth submanifold of ${\mathbb R}^{N+2}$ of codimension two.
We will need, however, the fact that $\mathcal{C}_0$ is a smooth submanifold of $\mathcal{C}$ of codimension one.
To verify this, we use the representation
\begin{equation}
  \mathcal{C}_0 = \{(\xi,\eta)\in \mathcal{C}\setminus\mathcal{C}_1\,|\, \Upsilon_0(\xi,\eta) = 0\},
  \label{C0C}
\end{equation}
which means that the set $\mathcal{C}_0$, considered as a subset of $\mathcal{C}$, is locally determined
as a level set of the restriction $\Upsilon_0|_\mathcal{C}$.
Therefore, it suffices to check that 
\begin{equation}
  T_{(\xi,\eta)}^*\mathcal{C}\ni d(\Upsilon_0|_\mathcal{C})(\xi,\eta) \ne 0, \quad (\xi,\eta) \in \mathcal{C}_0. 
  \label{Ups0-nondeg}
\end{equation}
The contrary to this implies
\begin{equation*}
  d\Upsilon_0(\xi,\eta)\big|_{T_{(\xi,\eta)}\mathcal{C}} = d(\Upsilon_0|_\mathcal{C})(\xi,\eta) = 0.
\end{equation*}
Then the differential $d\Upsilon_0(\xi,\eta)$
vanishes on the subspace $T_{(\xi,\eta)}\mathcal{C} \subset T_{(\xi,\eta)}{\mathbb R}^{N+2}$ of codimension one.
Thus, by the definition of the manifold $\mathcal{C}$,
it is proportional to the differential $d\Upsilon(\xi,\eta)$.
We have seen, however, that these two differentials are linearly independent
on the set $\mathcal{C}_0$.

\begin{lemma}
  Under the assumption~(\ref{nontang}),
  the function $\Phi(\xi,\eta)$ in ${\mathbb R}^{N+2}$ satisfies
\begin{gather}
  d\Phi(\xi,\eta)\big|_{T_{(\xi,\eta)}\mathcal{C}_0} \ne 0, \quad (\xi,\eta)\in \mathcal{C}_0.
  \label{ddPhi-ne0}
  \\
  d\Phi(\xi,\eta)\big|_{T_{(\xi,\eta)}\mathcal{C}} \ne 0, \quad (\xi,\eta)\in \mathcal{C}\setminus\mathcal{C}_{\rm st}.
  \label{dPhi-ne-0}
\end{gather}
\end{lemma}
\begin{proof}
Elements of the space $T_{(\xi,\eta)}^* {\mathbb R}^{N+2}$ that vanish on $T_{(\xi,\eta)}\mathcal{C}_0$
form a two-dimensional subspace.
As was shown before, the differentials $d\Upsilon$, $d\Upsilon_0$ form a basis in the latter.
So the contrary to relation~(\ref{ddPhi-ne0}) means that for some $\alpha, \alpha_0\in{\mathbb R}$, we have
\begin{equation*}
  d\Phi(\xi,\eta) = \alpha d\Upsilon(\xi,\eta) + \alpha_0 d\Upsilon_0(\xi,\eta).
\end{equation*}
Then, in view of equalities
\begin{gather*}
  \partial_{\xi_0} \Phi = \theta_0, \quad \partial_{\eta_0} \Phi = -\omega_0,
  \\
  \partial_{\xi_0} \Upsilon = 2 \xi_0,
  \quad
  \partial_{\eta_0} \Upsilon = -2 \eta_0,
  \\
  \partial_{\xi_0} \Upsilon_0 = 1,
  \quad
  \partial_{\eta_0} \Upsilon_0 = 1,
\end{gather*}
we deduce that
\begin{equation*}
  \theta_0 = 2\alpha \xi_0 + \alpha_0,
  \quad
  -\omega_0 = -2\alpha \eta_0 + \alpha_0.
\end{equation*}
Taking into account the equality $\xi_0 + \eta_0 = 0$ satisfied on $\mathcal{C}_0$,
we come to a contradiction with assumption~(\ref{nontang}).

Now turn to the proof of~(\ref{dPhi-ne-0}).
The contrary to the latter means that
\begin{equation*}
  d\Phi(\xi,\eta) = \alpha d\Upsilon(\xi,\eta)
\end{equation*}
for some $\alpha\in{\mathbb R}$.
This can be written in the form of the equalities 
\begin{equation*}
  \partial_\xi\Phi(\xi,\eta) = 2 \alpha \xi, \quad \partial_\eta\Phi(\xi,\eta) = -2\alpha\eta,
\end{equation*}
which are equivalent to
\begin{equation*}
  \theta = 2 \alpha \xi, \quad \omega = 2 \alpha \eta.
\end{equation*}
Hence $\alpha\ne 0$, and the point $(\xi,\eta)$ belongs to $\mathcal{C}_{\rm st}$, 
which contradicts the condition of Lemma.
\end{proof}

In the following lemmas we obtain certain estimates of functions $a^\beta_j$.
Henceforth we will not indicate explicitly that the constants in estimates depend on
$\theta$, $\omega$, $R$, $\beta$, $\chi$, $v_0$.

\begin{lemma}
  For $(\zeta,\sigma,r)\in\Sigma\times{\mathbb R}_+$, we have
\begin{equation}
  |\partial_r^k (a^\beta_j(r\zeta, r\sigma))| \leqslant C_k 2^{j(1+|\beta|)}, \quad k\geqslant 0, \quad j\in{\mathbb Z}.
  \label{aj-r-est}
\end{equation}
\end{lemma}
\begin{proof}
Since $P^\beta(2^j r\zeta, 2^j r\sigma) = (2^j r)^{|\beta|} P^\beta(\zeta,\sigma)$,
by definitions~(\ref{ajbe}), (\ref{abe}), (\ref{uhat}), we have
\begin{equation}
  a^\beta_j(r\zeta, r\sigma) = 
  2\pi (2^j r)^{1+|\beta|} P^\beta(\zeta,\sigma) \tilde\chi(r) |\zeta_0 + \sigma_0|\, \tilde v_0(2^j r (\zeta_0 + \sigma_0), 2^j r\overline\zeta, 2^j r\overline\sigma).
  \label{ajbe-full}
\end{equation}
Hence it suffices to show that
\begin{equation}
  |(\zeta_0 + \sigma_0)\, \partial_r^k(\tilde v_0(2^j r (\zeta_0 + \sigma_0), 2^j r\overline\zeta, 2^j r\overline\sigma))| \leqslant C_k, \quad k\geqslant 0,
  \quad r\in{\rm supp}\tilde\chi.
  \label{psi-deriv-est}
\end{equation}
The expression under the modulus sign equals
\begin{equation*}
  \sum_{\alpha+|\nu|+|\gamma|=k} C_{\alpha,\nu,\gamma} 2^{k j} (\zeta_0 + \sigma_0)^{\alpha+1}\, \overline\zeta^\nu\, \overline\sigma^\gamma
  \left(\partial_{\lambda,\overline\xi,\overline\eta}^{\alpha,\nu,\gamma}\, \tilde v_0\right)(2^j r (\zeta_0 + \sigma_0), 2^j r\overline\zeta, 2^j r\overline\sigma).
\end{equation*}
We have
\begin{equation*}
  |(\zeta_0 + \sigma_0)^{\alpha+1}\, \overline\zeta^\nu\, \overline\sigma^\gamma| 
  \leqslant C_k |\zeta_0 + \sigma_0|^{\alpha+1} \cdot |\overline\zeta|^{|\nu|} \cdot |\overline\sigma|^{|\gamma|}
  \leqslant C_k \rho^{k+1} \leqslant C_k \rho^k,
\end{equation*}
where $\rho := |\zeta_0 + \sigma_0| + |\overline\zeta| + |\overline\sigma|$.
At the same time, since $\tilde v_0\in \mathcal{S}({\mathbb R}^{N+1})$,
the partial derivative in the expression above is majorized by
\begin{equation*}
  C_{k,M} \left(1 + 2^j \rho\right)^{-M}, \quad M \geqslant 0 
\end{equation*}
(we took into account that the variable $r\in{\rm supp}\tilde\chi$ is separated from zero).
Letting $M = k$ here, we estimate the whole expression by
\begin{equation*}
  \frac{C_k 2^{k j} \rho^k}{\left(1 + 2^j \rho\right)^k}
  \leqslant C_k.
\end{equation*}
This implies estimate~(\ref{psi-deriv-est}) and thus~(\ref{aj-r-est}) holds true.
\end{proof}

For a function $f(\xi,\eta)$ in ${\mathbb R}^{N+2}$, a fixed point $(\xi,\eta)\in{\mathbb R}^{N+2}$, and an integer $k\geqslant 0$,
we set
\begin{equation*}
  |f|_{\xi,\eta,k} = \max_{|\alpha|\leqslant k} |\partial^\alpha f(\xi,\eta)|.
\end{equation*}
Note the evident inequality
\begin{equation}
  |f g|_{\xi,\eta,k} \leqslant C_\mu |f|_{\xi,\eta,k} \cdot |g|_{\xi,\eta,k}.
  \label{prod-est}
\end{equation}

\begin{lemma}
  For $(\xi,\eta)\in{\mathbb R}^{N+2}$, $\xi_0 + \eta_0 \ne 0$, $j\in{\mathbb Z}$, $M\geqslant 0$, we have
\begin{equation}
  |a^\beta_j|_{\xi,\eta,k}
  \leqslant \frac{C_{k, M} 2^{(1+|\beta|)j} (1 + 2^{k j})}{(1 + 2^j (|\xi_0 + \eta_0| + |\overline\xi| + |\overline\eta|))^M}.
  \label{aj-full-est}
\end{equation}
\end{lemma}
\begin{proof}
According to definitions~(\ref{ajbe}), (\ref{abe}), (\ref{uhat}), we have
\begin{equation*}
  a^\beta_j(\xi,\eta) = 2\pi\, 2^{(1+|\beta|)j} \chi(\xi,\eta) P^\beta(\xi,\eta) |\xi_0 + \eta_0|\, \tilde v_0(2^j(\xi_0 + \eta_0), 2^j \overline\xi, 2^j \overline\eta)
\end{equation*}
(this equality is analogous to~(\ref{ajbe-full})).
Put 
\begin{equation*}
  \widetilde P^\beta(\xi,\eta) := \chi(\xi,\eta) P^\beta(\xi,\eta) |\xi_0 + \eta_0|
\end{equation*}
and apply inequality~(\ref{prod-est}):
\begin{equation*}
  |a^\beta_j|_{\xi,\eta,k}
  \leqslant C_k\,2^{(1+|\beta|)j} |\widetilde P^\beta|_{\xi,\eta,k}
  \cdot \max_{|\alpha|\leqslant k} |\partial^\alpha (\tilde v_0(2^j(\xi_0 + \eta_0), 2^j \overline\xi, 2^j \overline\eta))|.
\end{equation*}
Next,
\begin{multline*}
  |\partial^\alpha (\tilde v_0(2^j(\xi_0 + \eta_0), 2^j \overline\xi, 2^j \overline\eta))|
  = 2^{|\alpha| j} |(\partial^\alpha \tilde v_0)(2^j(\xi_0 + \eta_0), 2^j \overline\xi, 2^j \overline\eta))|
  \\\leqslant \frac{C_{\alpha, M} 2^{|\alpha| j}}{(1 + 2^j (|\xi_0 + \eta_0| + |\overline\xi| + |\overline\eta|))^M},
\end{multline*}
so the maximum over $\alpha$ in the preceding expression does not exceed
\begin{equation*}
  \frac{C_{k, M} (1 + 2^{k j})}{(1 + 2^j (|\xi_0 + \eta_0| + |\overline\xi| + |\overline\eta|))^M}.
\end{equation*}
Then, since $|\widetilde P^\beta|_{\xi,\eta,k} \leqslant C_{\beta,k}$,
we arrive at inequality~(\ref{aj-full-est}).
\end{proof}

We will apply estimate~(\ref{aj-full-est}) only in the case,
when a point $(\xi,\eta)$ is separated from $\mathcal{C}_1$.
Namely, let $\mathcal{K}$ be a compact subset of $\mathcal{C}\setminus\mathcal{C}_1$.
It follows from the definition of the set $\mathcal{C}_1$ that
for some $\varepsilon_\mathcal{K}>0$, we have
\begin{equation*}
  |\xi_0 + \eta_0| + |\overline\xi| + |\overline\eta| \geqslant \varepsilon_\mathcal{K}, \quad (\xi,\eta)\in\mathcal{K}.
\end{equation*}
Combined with estimate~(\ref{aj-full-est}), this yields
\begin{equation}
  |a^\beta_j|_{\xi,\eta,k}
  \leqslant \frac{C_{k, M, \mathcal{K}} 2^{(1+|\beta|)j}}{(1 + 2^j)^M}, \quad (\xi,\eta)\in\mathcal{K}.
  \label{aj-est}
\end{equation}

\section{Estimates of the integrals $I^\beta_{\varkappa, j}$}
Henceforth we will denote integrals $I^\beta_j(x,y)$ and $I^\beta_{\varkappa, j}(x,y)$
by $I^\beta_j(\tau)$ and $I^\beta_{\varkappa, j}(\tau)$, respectively, in the case of
a point $(x,y)$ given by~(\ref{xy-tending}). 

In this section, we will obtain certain estimates of the integrals $I^\beta_{\varkappa, j}(\tau)$
and inspect their asymptotic behavior as $\tau\to+\infty$,
depending on the choice of the cut-off function $\varkappa$ on $\mathcal{C}$.
In Lemmas~\ref{lemma-cover}, \ref{lemma-cover2}, \ref{lemma-cover3},
we assume that the support of the latter is separated from $\mathcal{C}_{\rm st}$,
in which case the following estimate will be established
\begin{equation}
  |I^\beta_{\varkappa, j}(\tau)| \leqslant \frac{C_{\varkappa, \tau_0}}{2^{N j + |j|} \tau^{N + |\beta|}}, \quad \tau\geqslant\tau_0>0, \quad j\in{\mathbb Z}.
  \label{Ikj-est}
\end{equation}
(Recall that the dependence of constants on $\theta$, $\omega$, $R$, $\beta$, $\chi$, $v_0$ in estimates
is not indicated explicitly.)
Lemma~\ref{lemma-stat} concerns the case when the support of the function $\varkappa$
intersects with $\mathcal{C}_{\rm st}$. 

Everywhere in the text all of the neighborhoods are assumed to be open sets.

\begin{lemma}\label{lemma-cover}
Under the assumptions of Theorem~\ref{thm},
every point $(\xi^*,\eta^*)\in\mathcal{C}_1$ has a neighborhood $V\subset\mathcal{C}$, depending on $\theta,\omega,\tau_0$, and $R$
such that~(\ref{Ikj-est}) holds true for arbitrary function $\varkappa\in C_0^\infty(V)$.
\end{lemma}
\begin{proof}
According to~(\ref{xy-tending}), we have
\begin{equation}
  x\xi - y\eta = X \xi - Y\eta + \tau (\theta\xi - \omega\eta).
  \label{xxiyeta}
\end{equation}
It follows from condition~(\ref{char-init}) that
\begin{equation*}
  X \xi - Y\eta
  = X_0 \xi_0 - Y_0 \eta_0 + \overline X \,\overline\xi - \overline Y \overline\eta
  = X_0 (\xi_0 + \eta_0) + \overline X \,\overline\xi - \overline Y \overline\eta.
\end{equation*}
Be definition of $\mathcal{C}_1$, we have
\begin{gather}
  \xi^*_0 = -\eta^*_0 \ne 0,  \label{xiet-star1}
  \\
  \overline{\xi^*} = 0 = \overline{\eta^*}. \label{xiet-star2}
\end{gather}
Hence the preceding expression vanishes at the point $(\xi,\eta) = (\xi^*, \eta^*)$,
and does not exceed $4 R (|\xi-\xi^*| + |\eta-\eta^*|)$ in absolute value at arbitrary point.
Therefore, for arbitrary $\varepsilon>0$, the estimate
\begin{equation}
  |X \xi - Y\eta| \leqslant \tau_0 \varepsilon 
  \label{XYxieta}
\end{equation}
holds in a sufficiently small neighborhood of the point $(\xi^*, \eta^*)$.

Now turn to the term $\tau (\theta\xi - \omega\eta)$ in~(\ref{xxiyeta}).
In view of equality~(\ref{xiet-star1}), we have
\begin{equation*}
  \theta\xi - \omega\eta
  = (\theta_0 + \omega_0) \xi_0^* + \theta_0 (\xi_0 - \xi_0^*) - \omega_0(\eta_0-\eta_0^*) + \overline\theta\,\overline\xi - \overline\omega\,\overline\eta.
\end{equation*}
According to~(\ref{xiet-star2}), the sum of the last three terms here
vanishes at the point $(\xi,\eta) = (\xi^*, \eta^*)$,
and thus does not exceed $2\varepsilon$ in absolute value in a sufficiently small neighborhood of this point.
Now letting
\begin{equation*}
  \varepsilon = |(\theta_0 + \omega_0) \xi^*_0|/4,
\end{equation*}
($\varepsilon > 0$ in view of conditions~(\ref{nontang}) and~(\ref{xiet-star1})),
we obtain the estimate $|\theta\xi - \omega\eta| \geqslant 2\varepsilon$.
Combined with estimate~(\ref{XYxieta}) and condition $\tau\geqslant\tau_0$, this implies that 
\begin{equation}
  |x\xi - y\eta| \geqslant 
  \tau \varepsilon, \quad (\xi,\eta)\in V,
  \label{Psi-ne-0}
\end{equation}
where $V$ is a sufficiently small neighborhood of the point $(\xi^*, \eta^*)$.

Introduce the differential operator
\begin{equation*}
  L = (x\zeta - y\sigma)^{-1} \partial_r,
\end{equation*}
which is well-defined on the part of the manifold $\Sigma\times{\mathbb R}_+$, in which the denominator of the expression is nonzero.
The counterpart of the operator $L$ induced on $\mathcal{C}$ by diffeomorphism~(\ref{diffeo})
will be denoted by the same symbol. 
By~(\ref{Psi-ne-0}), for $(\xi,\eta) = (r\zeta,r\sigma) \in V$, we have
\begin{equation}
  |x\zeta - y\sigma| = r^{-1} |x\xi - y\eta| \geqslant r^{-1} \varepsilon\tau.
  \label{denom-ne-0}
\end{equation}
This quantity is positive, 
so the operator $L$ is well-defined in $V$.
Let $L_T$ denote the differential operator formally adjoint to $L$
with respect to the volume form~(\ref{dS-coord}).
We have
\begin{equation}
  L_T^k = \frac{(r^{1-N} \partial_r (r^{N-1} \cdot))^k}{(y\sigma - x\zeta)^k}, \quad k\geqslant 1.
  \label{L-adj}
\end{equation}

Observe that
\begin{equation*}
  e^{i 2^j (x\xi - y\eta)} = (i 2^j)^{-k} L^k e^{i 2^j (x\xi - y\eta)},
\end{equation*}
whence
\begin{equation*}
  |I_{\varkappa,j}(\tau)| \leqslant 2^{-k j} \int_\mathcal{C} \left|\left(L_T^k (\varkappa a^\beta_j)\right)(\xi, \eta)\right| dS_{\xi,\eta}.
\end{equation*}
The integrand is nonzero at a point $(\xi,\eta) = (r\zeta,r\sigma)$ only if $r\in{\rm supp}\tilde\chi \subset {\mathbb R}_+$,
which implies that $r$ must be bounded and separated from zero.
Then by formula~(\ref{L-adj}) and estimate~(\ref{denom-ne-0}), we have
\begin{equation*}
  \left|\left(L_T^k (\varkappa a^\beta_j)\right)(r\zeta, r\sigma)\right|
  \leqslant C_{\varkappa,k}\, \tau^{-k} \max_{0\leqslant \alpha\leqslant k} |\partial_r^\alpha (a^\beta_j(r\zeta, r\sigma))|.   
\end{equation*}
According to~(\ref{aj-r-est}), the last expression is bounded by $C_{\varkappa,k} \tau^{-k} 2^{(1+|\beta|) j}$, whence
\begin{equation}
  |I_{\varkappa,j}(\tau)| \leqslant C_{\varkappa,k}\, \tau^{-k} 2^{(1+|\beta|-k) j}.
  \label{Ikj-b-mu}
\end{equation}
Now taking $k = N + 1 + |\beta| + {\rm sgn} j$, we arrive at estimate~(\ref{Ikj-est}).
\end{proof}

It is convenient to write the integral $I^\beta_{\varkappa,j}(\tau)$ 
as follows 
\begin{equation*}
  I^\beta_{\varkappa,j}(\tau) = \int_\mathcal{C} (\varkappa b^\beta_j)(\xi, \eta)
  e^{i \tau 2^j \Phi(\xi,\eta)} dS_{\xi,\eta},
\end{equation*}
where
\begin{equation}
  b^\beta_j(\xi,\eta) = a_j^\beta(\xi,\eta) e^{i 2^j (X \xi - Y \eta)}.
  \label{bj-def}
\end{equation}

The estimate of the function $b^\beta_j$ analogous to~(\ref{aj-est}) will be required.
To derive it, observe that for the function $e^{i 2^j (X \xi - Y \eta)}$ of variables $(\xi,\eta)$,
the functional $|\cdot|_{\xi,\eta,k}$ is bounded by $C_k (1 + 2^{k j})$.
Then, in view of~(\ref{prod-est}), the following inequality is valid
\begin{equation*}
  |b^\beta_j|_{\xi,\eta,k} \leqslant C_k (1 + 2^{k j}) |a^\beta_j|_{\xi,\eta,k}.
\end{equation*}
Now applying estimate~(\ref{aj-est}) gives
\begin{equation}
  |b^\beta_j|_{\xi,\eta,k} 
  \leqslant \frac{C_{k, M, \mathcal{K}} 2^{(1+|\beta|)j}}{(1 + 2^j)^M},
  \label{bj-est}
\end{equation}
as soon as $(\xi,\eta)$ belongs to a compact set $\mathcal{K} \subset \mathcal{C}\setminus\mathcal{C}_1$.

\begin{lemma}\label{lemma-cover2}
  Under the assumptions of Theorem~\ref{thm},
  every point $(\xi^*,\eta^*)\in\mathcal{C}_0$ has a neighborhood $V\subset\mathcal{C}$
  such that estimate~(\ref{Ikj-est}) is valid for an arbitrary function $\varkappa\in C_0^\infty(V)$.
\end{lemma}
\begin{proof}
Note that condition~(\ref{nontang}) will not be used in the proof.
Relation~(\ref{ddPhi-ne0}) implies that there exists a vector $z\in T_{(\xi^*,\eta^*)} \mathcal{C}_0$,
at which the value of the differential $d\Phi(\xi^*,\eta^*)$ is nonzero.
It is possible to construct a vector field $Z$ in some neighborhood $V\subset\mathcal{C}$ of the point $(\xi^*,\eta^*)$
that is equal to $z$ at $(\xi^*,\eta^*)$ and satisfies the following condition
\begin{equation}
  Z(\xi,\eta) \in T_{(\xi,\eta)} \mathcal{C}_0, \quad (\xi,\eta) \in V\cap\mathcal{C}_0.
  \label{Z-tang}
\end{equation}
Next we pass (if necessary) to a smaller neighborhood $V$, in which
the condition $\langle{}d\Phi, Z\rangle\ne 0$ is satisfied. 
This can be written as follows
\begin{equation}
  Z\Phi\big|_V \ne 0.
  \label{ZPhi}
\end{equation}
Since $(\xi^*,\eta^*)\notin\mathcal{C}_1$, we may assume that
\begin{equation}
  \overline{V} \cap \mathcal{C}_1 = \emptyset.
  \label{V-star-sep}
\end{equation}
According to~(\ref{C0C}), we have $\Upsilon_0(\xi^*,\eta^*) = 0$,
and, besides,
\begin{equation}
  V_+\cup V_- = V\setminus \mathcal{C}_0,
  \quad
  V_\pm := \{ (\xi,\eta) \in V \,|\, \pm \Upsilon_0(\xi,\eta) > 0\}.
  \label{VpmV}
\end{equation}
In view of~(\ref{Ups0-nondeg}), the sets $V_\pm$ are nonempty.
We may assume that $V$ is chosen in such a way that $V_\pm$
are domains with smooth boundaries on the manifold ${\mathbb C}$.
Now we show that
\begin{equation}
  V \cap \partial V_\pm \subset V \cap \mathcal{C}_0
  \label{V-inclus}
\end{equation}
(in reality, this is valid with the equal sign as well).
Let a point $(\xi, \eta)$ belong to $V \cap \partial V_\pm$.
This implies that $\pm\Upsilon_0(\xi,\eta) \geqslant 0$.
On the other hand, since the set $V_\pm$ is open, it does not contain its boundary point $(\xi,\eta)$,
and so
\begin{equation*}
  (\xi, \eta) \in V\setminus V_\pm \subset V_\mp \cup \mathcal{C}_0,
\end{equation*}
(the last inclusion follows from~(\ref{VpmV})).
Together with~(\ref{C0C}) and the definition of the set $V_\mp$ this leads to inequality $\mp\Upsilon_0(\xi,\eta) \geqslant 0$.
Thus $\Upsilon_0(\xi,\eta) = 0$, and so $(\xi,\eta) \in \mathcal{C}_0$.

In view of~(\ref{ZPhi}), the differential operator
\begin{equation*}
  L = (Z\Phi)^{-1} Z
\end{equation*}
is well-defined in the neighborhood $V$,
and the following relation is valid
\begin{equation*}
  e^{i \tau 2^j \Phi} = 
  (i \tau 2^j)^{-k} L^k e^{i \tau 2^j \Phi}.
\end{equation*}
Therefore, for any function $\varkappa\in C_0^\infty(V)$, we have
\begin{equation}
  I^\beta_{\varkappa, j}(\tau)
  =  (i \tau 2^j)^{-k}
  \int_V (\varkappa b^\beta_j L^k e^{i \tau 2^j \Phi})(\xi,\eta) dS_{\xi,\eta}.
  \label{I-L}
\end{equation}
By equality~(\ref{VpmV}), the integral here equals the sum of the integrals over $V_+$, $V_-$
(integrating over the set $V\cap \mathcal{C}_0$ gives zero since the latter has zero measure).
Next, the functions $b^\beta_j$ are $C^\infty$-smooth in $V_\pm$,
and so, in each of the two integrals over these sets, we may perform integrating by parts.
This yields the following relation
\begin{equation}
  I^\beta_{\varkappa, j}(\tau) 
  =  \sum_\pm(i \tau 2^j)^{-k}
  \int_{V_\pm} \left( e^{i \tau 2^j \Phi} L_T^k (\varkappa b^\beta_j) \right)(\xi,\eta) \,dS_{\xi,\eta},
  \label{I-Vpm}
\end{equation}
where $L_T$ is a differential operator formally adjoint to $L$ with respect to the volume form~(\ref{dS-coord}).
Note that the corresponding integral over the boundary $\partial V_\pm$ is zero.
Indeed, the integrand vanishes on the set $\partial V_\pm \setminus V$ due to the factor $\varkappa$,
and the intersection $V \cap \partial V_\pm$, by~(\ref{V-inclus}), is contained in the submanifold $V\cap \mathcal{C}_0$,
to which the field $Z$ is tangential by condition~(\ref{Z-tang}).

For $(\xi,\eta)\in V_\pm$, we have
\begin{equation*}
  \left|\left(L_T^k (\varkappa b^\beta_j)\right)(\xi,\eta)\right|
  \leqslant C_{\varkappa,k} |b^\beta_j|_{\xi,\eta,k}
  \leqslant C_{\varkappa,k} 2^{(1+|\beta|)j}.
\end{equation*}
In the second inequality, we took into account condition~(\ref{V-star-sep})
invoking also estimate~(\ref{bj-est}) with $M = 0$.
Then, according to~(\ref{I-Vpm}), we arrive at the estimate of the form~(\ref{Ikj-b-mu}).
It remains to make the same choice of $k$ as in the end of the proof of Lemma~\ref{lemma-cover}.
\end{proof}

\begin{lemma}\label{lemma-cover3}
  Under the assumptions of Theorem~\ref{thm},
  every point $(\xi^*,\eta^*)\in\mathcal{C}\setminus(\mathcal{C}_0\cup\mathcal{C}_1\cup\mathcal{C}_{\rm st})$
  has a neighborhood $V\subset\mathcal{C}$
  such that estimate~(\ref{Ikj-est}) is valid for any function $\varkappa\in C_0^\infty(V)$.
\end{lemma}
\begin{proof}
By relation~(\ref{dPhi-ne-0}), we may choose a vector field $Z$ in some neighborhood $V\subset\mathcal{C}$ of the point $(\xi^*,\eta^*)$
in such a way that condition of the form~(\ref{ZPhi}) is satisfied.
Thus we obtain the representation of the integral $I^\beta_{\varkappa,j}$ analogous to~(\ref{I-L}).
Since $(\xi^*,\eta^*)$ lies outside of the set $\mathcal{C}_0\cup\mathcal{C}_1$, we may assume that $\overline V$ is also separated from this set.
Then the factor $b^\beta_j$ in the integrand in~(\ref{I-L}) is smooth on the set of integration,
which allows us to perform integration by parts.
Further deduction of estimate~(\ref{Ikj-est}) is the same
as in the proof of Lemma~\ref{lemma-cover2},
but without 
decomposition of the neighborhood $V$ to $V_\pm$.
\end{proof}

In view of inclusion~(\ref{stat-ray}),
under the assumptions of Lemmas~\ref{lemma-cover}, \ref{lemma-cover2}, \ref{lemma-cover3},
the point $(\xi^*,\eta^*)$ lies outside of the set $\mathcal{C}_{\rm st}$.
So the neighborhoods $V$ constructed in these lemmas may also be assumed to lie outside of this set.

Now turn to investigation of the integral $I^\beta_{\varkappa, j}$ in the case
when the support of the function $\varkappa$ intersects with $\mathcal{C}_{\rm st}$.

\begin{lemma}   \label{lemma-stat}
  Under the assumptions of Theorem~\ref{thm},
  every point $(\xi^*,\eta^*)\in\mathcal{C}_{\rm st}$
  has a neighborhood $V\subset\mathcal{C}$
  such that for any $\varkappa\in C_0^\infty(V)$, $\tau>0$, $j\in{\mathbb Z}$, the following relations hold
\begin{equation}
  I^\beta_{\varkappa,j}(\tau) =
  \frac{(2\pi)^{N/2} e^{\pm i\pi (n-d)/4}}{2 (2^j \tau)^{N/2}}
  \int_0^\infty r^{N/2-1} (\varkappa b^\beta_j)(\pm r\theta, \pm r\omega) dr + K^\beta_{\varkappa,j}(\tau),
  \label{I-stat}
\end{equation}
\begin{equation}
  |K^\beta_{\varkappa,j}(\tau)| \leqslant \frac{C_{\varkappa, M} 2^{(1+|\beta| - (N+1)/2)j}}{\tau^{(N+1)/2} \, (1 + 2^j)^M}, \quad M \geqslant 0.
  \label{R-est}
\end{equation}
\end{lemma}
\begin{proof}
Under the given assumptions, we have
\begin{equation*}
  (\xi^*,\eta^*) = (\pm r^*\theta, \pm r^*\omega), \quad r^* > 0.
\end{equation*}
In view of condition~(\ref{nontang}), the point $(\pm\theta, \pm\omega)$ (here and further the sign $\pm$ is chosen the same
as in the preceding equality) has a neighborhood $\Sigma_{\rm st}$ on the manifold $\Sigma$ such that
\begin{equation*}
  \zeta_0 + \sigma_0 \ne 0, \quad (\zeta,\sigma)\in\Sigma_{\rm st}.
\end{equation*}
We choose the desired neighborhood of the point $(\xi^*, \eta^*)$ as follows
\begin{equation*}
  V = \{ (r\zeta,r\sigma)\,|\, (\zeta,\sigma)\in\Sigma_{\rm st}, \,\, r > 0 \}. 
\end{equation*}
The preceding condition ensures the inclusion
\begin{equation}
  V \subset \mathcal{C}\setminus(\mathcal{C}_0\cup \mathcal{C}_1).
  \label{VCC}
\end{equation}

For $\varkappa\in C_0^\infty(V)$, we have
\begin{equation}
  I^\beta_{\varkappa,j}(\tau) =
  \frac{1}{2} \int_0^\infty dr\, r^{N-1} \int_{\Sigma_{\rm st}} (\varkappa b^\beta_j)(r\zeta, r\sigma) e^{i \tau 2^j r \Phi(\zeta,\sigma)} dS_{\zeta,\sigma}.
  \label{I-princ}
\end{equation}
Consider the inner integral over $\Sigma_{\rm st}$.
In view of~(\ref{VCC}), the function $b^\beta_j$ is smooth on the set $V$.
Hence, for fixed $r$, the function
\begin{equation*}
  B_j(\zeta, \sigma, r) = (\varkappa b^\beta_j)(r\zeta, r\sigma)
\end{equation*}
is smooth in $\Sigma_{\rm st}$.
The asymptotics of such an integral can be found by the method of stationary phase,
which was done in~\cite{DMN23}.
All of the stationary points of the phase function $\Phi|_\Sigma$ 
\begin{equation*}
  (\zeta,\sigma) \in \{ (\theta,\omega), (-\theta,-\omega), (-\theta,\omega), (\theta,-\omega) \}
\end{equation*}
are nondegenerate.
We choose the set $\Sigma_{\rm st}$ in such a way that it contained only the point $(\pm\theta,\pm\omega)$ from the set above.
Then, for $\tau>0$, the integral of interest can be represented in the form
\begin{equation}
  \left(\frac{2\pi}{2^j r \tau}\right)^{N/2} 
  e^{\pm i\pi (n-d)/4} B_j(\pm\theta, \pm\omega, r)
  + \widetilde K^\beta_{\varkappa,j}(\tau, r),
  \label{sphere-asymp}
\end{equation}
where the remainder term satisfies the following estimate
\begin{equation}
  |\widetilde K^\beta_{\varkappa,j}(\tau, r)| \leqslant \frac{C \|B_j(\cdot, \cdot, r)\|_{C^{N+2}(\Sigma_{\rm st})}}{(2^j r \tau)^{(N+1)/2}}.
  \label{wti-K}
\end{equation}
By the norm in $C^{N+2}(\Sigma_{\rm st})$, we mean one of the equivalent norms in this space
provided by a finite atlas on the manifold $\Sigma$ and the corresponding partition of unity.
It can be easily seen, however, that the $C^{N+2}$-norm of the function $B_j(\zeta, \sigma, r)$
of variables $(\zeta,\sigma)\in\Sigma_{\rm st}$ is bounded for any $r$ by
\begin{equation*}
  C_\varkappa \left(1 + r^{N+2}\right) \max_{(\xi,\eta)\in r\Sigma_{\rm st}} |b^\beta_j|_{\xi,\eta,N+2}.
\end{equation*}
Invoking estimate~(\ref{bj-est}),
we may deduce from this and~(\ref{wti-K}) the following estimate
\begin{equation*}
  r^{N-1} |\widetilde K^\beta_{\varkappa,j}(\tau, r)|
  \leqslant \frac{C_{\varkappa, M} 2^{(1 + |\beta| - (N+1)/2)j} (1 + r^{N+2}) r^{(N-3)/2}}{\tau^{(N+1)/2} (1 + 2^j)^M}.
\end{equation*}
The factor that depends on $r$ can be absorbed by the constant $C_{\varkappa,M}$,
as the function $\varkappa$ is nonzero only if $r$ is bounded and separated from zero.
Thus the contribution of the function $\widetilde K^\beta_{\varkappa,j}(\tau, r)$ to the integral~(\ref{I-princ})
is bounded by the r.h.s. of~(\ref{R-est}).

It remains to note that the contribution of the first term of the sum~(\ref{sphere-asymp})
to the integral~(\ref{I-princ}) equals the first term of the r.h.s. of~(\ref{I-stat}).
\end{proof}

\section{Proof of Theorem~\ref{thm}}   \label{sec-proof}
Every point $(\xi^*,\eta^*)$ of the manifold $\mathcal{C}$
has a neighborhood $V\subset\mathcal{C}$ for which 
the assertion of one of the Lemmas~\ref{lemma-cover}-\ref{lemma-stat} is valid.
Suppose that the chosen neighborhoods are precompact.
They form an open cover of the manifold $\mathcal{C}$.
We choose its locally finite refinement (which does exist, as $\mathcal{C}$ is a paracompact space).
Next we choose a smooth partition of unity on $\mathcal{C}$ subordinate to the resulting cover.
It consists of compactly supported functions, since all of the sets from the cover are precompact.
Let $\{V_k\}$ denote the (finite) collection of sets from the cover
that intersect with the compact $\mathcal{K} := \mathcal{C} \cap {\rm supp}\chi$.
The collection of corresponding functions from the partition of unity will be denoted by $\{\varkappa_k\}$.
By construction we have
\begin{equation}
  \sum_k \varkappa_k\Big|_\mathcal{K} = 1.
  \label{partition}
\end{equation}
Let us split the collection $\{V_k\}$ in $\{V'_\ell\}$ and $\{V''_m\}$,
the former consisting of those sets that are contained in the neighborhoods
constructed with the use of Lemma~\ref{lemma-stat},
and the latter being the remainder.
We split the collection of functions $\{\varkappa_k\}$ in $\{\varkappa'_\ell\}$, $\{\varkappa''_m\}$ in the same way,
and put
\begin{equation*}
  \varkappa' = \sum_{\ell} \varkappa'_\ell, \quad \varkappa'' = \sum_{m} \varkappa''_m.
\end{equation*}
The integrals $I^\beta_{\varkappa',j}(\tau)$, $I^\beta_{\varkappa'',j}(\tau)$ can be represented as follows
\begin{equation*}
  I^\beta_{\varkappa',j}(\tau) = \sum_\ell I^\beta_{\varkappa'_\ell,j}(\tau), \quad
  I^\beta_{\varkappa'',j}(\tau) = \sum_m I^\beta_{\varkappa''_m,j}(\tau).
\end{equation*}
All of the terms of the last sum, and thus the integral $I^\beta_{\varkappa'',j}(\tau)$, satisfy estimate~(\ref{Ikj-est}).
Furthermore, the factor $\tau^{-(N + |\beta|)}$ can be replaced by $\tau^{-(N+1)/2}$, as $\tau \geqslant \tau_0$ and $N\geqslant 2$
(the latter follows from the condition $d,n\geqslant 1$). We arrive at the following estimate
\begin{equation}
  |I^\beta_{\varkappa'', j}(\tau)| 
  \leqslant \frac{C_{\tau_0}}{2^{N j + |j|} \tau^{(N+1)/2}}, \quad \tau\geqslant\tau_0>0.
  \label{Ikk-est}
\end{equation}
Next, equality~(\ref{partition}) implies
\begin{equation}
  (\varkappa' + \varkappa'')|_\mathcal{K} = 1,
  \label{partition-2}
\end{equation}
whence
\begin{equation}
  I^\beta_j(\tau) = I^\beta_{\varkappa',j}(\tau) + I^\beta_{\varkappa'',j}(\tau).
  \label{III}
\end{equation}
In view of the remark, made after the proof of Lemma~\ref{lemma-cover3},
we may assume that neighborhoods $\{V''_m\}$ do not intersect with $\mathcal{C}_{\rm st}$,
which implies that the function $\varkappa''$ vanishes in this set.
Together with equality~(\ref{partition-2}), this means that
\begin{equation*}
  \sum_{\ell} \varkappa'_\ell(\pm r\theta, \pm r\omega) = \varkappa'(\pm r\theta, \pm r\omega) = 1, \quad r\in {\rm supp}\tilde\chi.
\end{equation*}
Therefore summing equalities~(\ref{I-stat}) for the functions $\varkappa'_\ell$ with respect to $\ell$
yields the following representation
\begin{equation*}
  I^\beta_{\varkappa',j}(\tau) = F^\beta_j(\tau) + K^\beta_{\varkappa',j}(\tau),
\end{equation*}
in which
\begin{equation}
  F^\beta_j(\tau) = \sum_\pm \frac{(2\pi)^{N/2} \epsilon^{\pm 1}}{2 (2^j \tau)^{N/2}}
  \int_0^\infty r^{N/2-1} b^\beta_j(\pm r\theta, \pm r\omega) dr,
  \quad \epsilon := e^{i\pi (n-d)/4},
  \label{Fbj}
\end{equation}
and the term $K^\beta_{\varkappa',j}(\tau)$ satisfies the estimate of the form~(\ref{R-est}).
Now we may write equality~(\ref{III}) in the form
\begin{equation}
  I^\beta_j(\tau) = F^\beta_j(\tau) + K^\beta_{\varkappa',j}(\tau) + I^\beta_{\varkappa'',j}(\tau)
  \label{IJK}
\end{equation}
and estimate the contribution of each term on the r.h.s. to the series~(\ref{sum-CC}).
By estimate~(\ref{R-est}), in which we take $M > |\beta| + (N+1)/2$,
and estimate~(\ref{Ikk-est}), we have
\begin{equation}
  \sum_{j\in{\mathbb Z}} 2^{N j} (|K^\beta_{\varkappa',j}(\tau)| + |I^\beta_{\varkappa'',j}(\tau)|)
  \leqslant \frac{C_{\tau_0}}{\tau^{(N+1)/2}}.
  \label{KI-est}
\end{equation}
Now turn to the term $F^\beta_j(\tau)$.
We have
\begin{equation*}
  \sum_\pm \epsilon^{\pm 1} \int_0^\infty r^{N/2-1} b^\beta_j(\pm r\theta, \pm r\omega) dr
  = \int_{\mathbb R} \epsilon^{\,{\rm sgn}\, r} |r|^{N/2-1} b^\beta_j(r\theta, r\omega) dr.
\end{equation*}
By definition~(\ref{bj-def}), the resulting intergal equals
\begin{multline*}
  \int_{\mathbb R} \epsilon^{\,{\rm sgn}\, r} |r|^{N/2-1} \tilde\chi(|r|) \, a^\beta(2^j r\theta, 2^j r\omega) e^{i 2^j r (X \theta - Y \omega)} dr
  \\= 2^{-N j/2} \int_{\mathbb R} \epsilon^{\,{\rm sgn}\, r} |r|^{N/2-1} \tilde\chi(2^{-j} |r|) \, a^\beta(r\theta, r\omega) e^{i r (X\theta - Y\omega)} dr.
\end{multline*}
Now, by definitions~(\ref{abe}), (\ref{uhat}), we have
\begin{multline*}
  a^\beta(r\theta, r\omega) = (P^\beta a)(r\theta, r\omega) = r^{|\beta|} P^\beta(\theta, \omega) \, a(r\theta, r\omega)
  \\= 2\pi ({\rm sgn}\, r)^{|\beta|} |r|^{|\beta|+1} P^\beta(\theta,\omega) |\theta_0 + \omega_0|\, \widetilde v_0(r(\theta_0 + \omega_0), r \overline\theta, r \overline\omega).
\end{multline*}
Thus equality~(\ref{Fbj}) can now be written as follows
\begin{equation*}
  \frac{2^{N j} F^\beta_j(\tau)}{(2\pi)^{N+2}} =
  \frac{P^\beta(\theta,\omega) |\theta_0 + \omega_0|}{\tau^{N/2}}
  \int_{\mathbb R} e^{i r p} G^\beta_{d,n}(r) \tilde\chi(2^{-j} |r|)
  \widetilde v_0(r(\theta_0 + \omega_0), r \overline\theta, r \overline\omega) dr.
\end{equation*}
The integral in the resulting expression coincides with that in the formula for $F^\beta$ given in
the formulation of Theorem~\ref{thm}, up to the factor $\tilde\chi(2^{-j} |r|)$ in the integrand.
This factor turns to unity after summing with respect to $j$, which follows from equality~(\ref{sum-chi}).
Thus the contribution of the quantities $F^\beta_j(\tau)$ to the series~(\ref{sum-CC}) equals
\begin{equation}
  \sum_{j\in{\mathbb Z}} (2\pi)^{-N-2} 2^{N j} F^\beta_j(\tau) = \tau^{-N/2} F^\beta(\theta, \omega, p),
  \label{sum-J}
\end{equation}
and the series of the absolute values satisfies the following estimate
\begin{equation}
  \sum_{j\in{\mathbb Z}} (2\pi)^{-N-2}\, 2^{N j} |F^\beta_j(\tau)|
  \leqslant C \tau^{-N/2} \int_{\mathbb R} |r|^{N/2+|\beta|} |\widetilde v_0(r(\theta_0 + \omega_0), r \overline\theta, r \overline\omega)| dr.
  \label{J-est}
\end{equation}
In the derivation of the last two relations, we used to the properties~(\ref{sum-chi}) of the function $\tilde\chi$
and fast decay of the function $\widetilde v_0$.

Estimates~(\ref{KI-est}), (\ref{J-est}) imply that the series on the r.h.s. of equivalent relations~(\ref{u-xy}) and (\ref{sum-CC})
are absolutely convergent locally uniformly with respect to $\tau>0$, $X,Y$.
In view of remarks made before the formulation of Theorem~\ref{thm},
this means that these series are absolutely convergent locally uniformly with respect to $(x,y)$
outside of the hyperplane~(\ref{char-plane}).
As was shown in sec.~\ref{sec-little-paley}, this ensures the validity of formula~(\ref{u-xy})
or, equivalently, formula~(\ref{sum-CC}).
The latter, combined with equalities~(\ref{IJK}), (\ref{sum-J}), and estimate~(\ref{KI-est}),
leads to relation~(\ref{thm-asymp}).

\end{document}